\documentclass[10pt,oneside]{amsart}
\usepackage{amssymb, amscd, amsmath, amsthm, latexsym, hyperref,amsrefs,enumitem}
\usepackage[text={6.5in,9.0in},centering,letterpaper,dvips]{geometry}

\usepackage{enumitem}\setlist[enumerate,1]{font=\upshape, itemsep=.5ex}\setlist[itemize,1]{font=\upshape, itemsep=.5ex}
\usepackage{pinlabel}
\usepackage{caption}
\usepackage{xcolor}

\def\Z{{\mathbb Z}}
\def\F{{\mathbb F}}
\def\Q{{\mathbb Q}}
\def\R{{\mathbb R}}
\def\C{{\mathbb C}}

\def\A{{\mathbb A}}

\def\calq{\mathcal{Q}}

\def\calp{\mathcal{P}}
\def\calc{\mathcal{C}}

\def\calo{\mathcal{O}}

\newcommand{\compactlist}{\begin{list}{\enumerate}{\setlength{\leftmargin}{1em}}}

\def\cs{\mathbin{\#}} 
\def\co{\colon\thinspace}

\newtheorem{theorem}{Theorem} 
\newtheorem{lemma}[theorem]{Lemma}

\theoremstyle{definition}

\AtBeginDocument{%
\def\MR#1{}
}
\begin{document}
\title{Branched covers and rational homology balls}

\author{Charles Livingston}
\thanks{This work was supported by a grant from the National Science Foundation, NSF-DMS-1505586.   }
\address{Charles Livingston: Department of Mathematics, Indiana University, Bloomington, IN 47405}\email{livingst@indiana.edu}

%\date{\today}

%%%%%%%ABSTRACT%%%%%%%%%%%%%%

\begin{abstract} 
The concordance group of knots in $S^3$ contains a subgroup isomorphic to $(\Z_2)^\infty$, each element of which is represented by a  knot $K$ with the property that for every $n>0$, the   $n$--fold cyclic   cover of $S^3$ branched over $K$ bounds a rational homology ball.   This implies that the kernel of the  canonical homomorphism from the knot concordance group to the infinite direct sum of rational homology cobordism groups (defined via prime-power branched covers) contains an infinitely generated two-torsion subgroup.
\end{abstract}

\maketitle

%%%%%%%Section%%%%%%%%%%%%%%

\section{Introduction} \label{sec:intro} 

There is a homomorphism \[
\varphi     \colon   \calc  \to \prod _{q \in {\calq}}\Theta ^3_{\Q},
\]
where $\calc$ is the smooth concordance group of knots in $S^3$, $\calq$ is the set of prime power integers, and $\Theta ^3_{\Q}$ is the rational homology cobordism group.  For a knot $K$ and $q \in \calq$, the $q$--component  of  $\varphi(K)$ is the class of $M_q(K)$, the $q$--fold  cyclic cover of  $S^3$ branched over $K$.   

In the paper~\cite{MR4357613},  Aceto,  Meier, A.~Miller, M.~Miller,   Park, and    Stipsicz proved that $\ker \varphi$  contains a subgroup isomorphic to    $ (\Z_2)^5$.   Here we will prove that $\ker \varphi$ contains a  subgroup isomorphic to     $ (\Z_2)^\infty$.   Our examples are of the form $K \cs -K^r$, where $-K$ denotes  the concordance inverse of $K$ (the mirror image of $K$ with string orientation reversed),  and $K^r$  is formed from $K$ by reversing its string orientation.  Such knots are easily seen to be in the kernel of $\varphi$;  the more difficult work is to  find nontrivial examples of order two.
 
The first  known example of a nontrivial element in $\ker \varphi$ was represented by the  knot $ K_1 =  8_{17} \cs - 8_{17}^r$, which is   of order two in $\calc$.  That $K_1 $  is of order at most two is elementary;    that $K_1$ is nontrivial in $\calc$ is one of the main results of~\cite{MR1670424}, proved using twisted Alexander polynomials.

The  results of~\cite{2019arXiv190412014K}   provide an infinitely generated free subgroup of  $\ker \varphi$.    Conjecturely,  $\calc \cong \Z^\infty \oplus (\Z_2)^\infty$; if true,  then  $\ker \varphi \cong \Z^\infty \oplus (\Z_2)^\infty$.

\subsection{Main result.}  Figure~\ref{fig:companion} illustrates a knot $P_n$ in a solid torus, where $J_n$ represents the  braid  illustrated on the right in the case of $n=5$; $n$ will always be odd.    We let $K_n$ denote the satellite of $8_{17}$ built from $P_n$.  In standard notation, $K_n = P_n(8_{17})$.  For future reference, we illustrate the braid $J_n^*$ formed by rotating $J_n$ around the vertical axis.  
%\rotatebox{45}
%\reflectbox 

\begin{figure}[h]
\labellist
\pinlabel {\text{{\Large$P_n$}}} at -10 165
\pinlabel {\text{{\Large$J_n$}}} at 170 106
\pinlabel {\text{{\Large$J_n$}}} at  460  95
\pinlabel {\text{{\Large$J_n^*$}}} at  740  95
\endlabellist
\includegraphics[scale=.45]{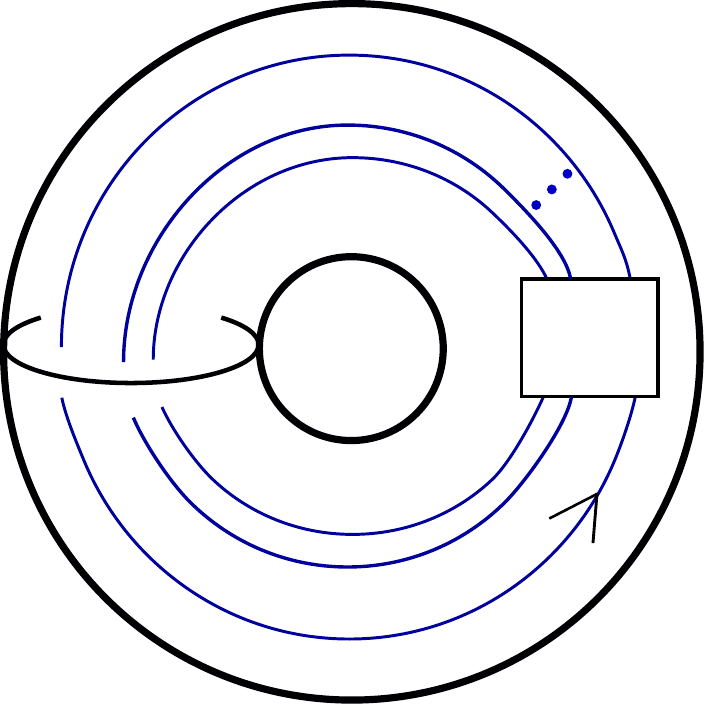}\hskip.5in  
\raisebox{.105in}{ \includegraphics[scale=.55]{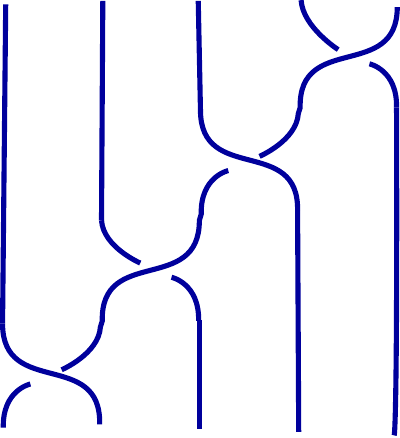}}\hskip.8in 
\raisebox{.15in}{ \includegraphics[scale=.55]{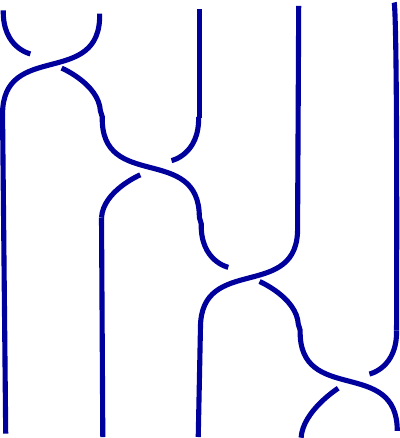}}
\caption{The knot $P_n \subset S^1 \times B^2$ , $J_n$, and $J_n^*$.}
\label{fig:companion}
\end{figure}

\begin{figure}[h]
\labellist
\pinlabel {\text{\huge $K$}} at 115 80
\pinlabel {\text{\huge$K$}} at 370  80
\endlabellist
\includegraphics[scale=.60]{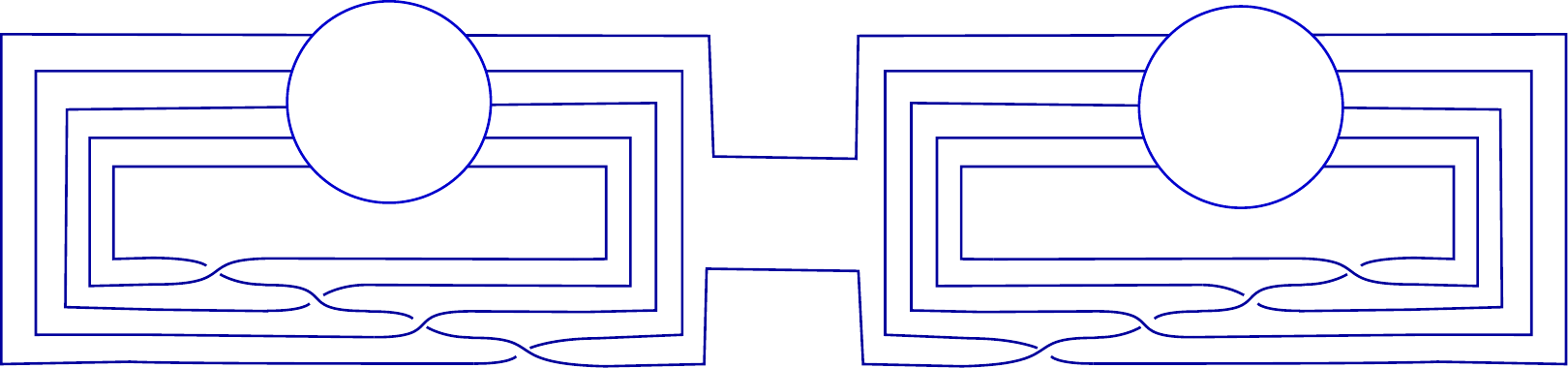}
\caption { $P_5(K) \cs    P_5(K)$.}
\label{fig:connectedsum}
\end{figure}

\begin{theorem}\label{thm:main} Let $K_n = P_n(8_{17})$.  For all odd $n$, the knot $L_n = K_n \cs -K_n^r$    satisfies $2L_n = 0 \in \calc$ and $L_n \in \ker \varphi$. There is an infinite set of prime integers $\calp$ for which   $L_\alpha \ne L_\beta \in \calc$  for all $\alpha \ne \beta$ in $\calp$.  In particular, the set of knots $\{L_n\}_{n \in \calp}$ generates a subgroup of $\ker \varphi$ that is  isomorphic to $(\Z_2)^\infty $.
\end{theorem}

The rest of the paper presents a proof of this theorem.  The first two claims are easily dealt with    in Sections~\ref{sec:ordertwo} and~\ref{sec:inthekernel}.   The more difficult step of the proof calls on an analysis of twisted Alexander polynomials and their relevance to knot slicing; a review of twisted polynomials is included  in Section~\ref{sec:twisted-polys}. 
The proof of Theorem~\ref{thm:main} is completed in Section~\ref{sec:mainproof}, with the exception of a number theoretic result that is described Appendix~\ref{app:numtheory}. 

\smallskip
\noindent{\it Acknowledgements.}  Thanks to Darrell Haile for assisting me in  proof of the  number theoretic result in  Appendix~\ref{app:numtheory}.  Allison Miller provided valuable feedback about an early draft of this paper.

%%%%%%%Section%%%%%%%%%%%%%%

\section {Proof:  $2L_n = 0 \in \calc$} \label{sec:ordertwo}

Let $P_n^* \subset S^1 \times B^2$ denote the knot formed  using the braid $J_n^*$ in 
Figure~\ref{fig:connectedsum}.  For any knot $K$, let $P_n^*(K)$ denote the satellite of $K$ built using $P_n^*$.   It should be clear that $P_n$ and $P_n^*$ are orientation preserving isotopic, and thus for all knots $K$, $P_n(K) = P_n^*(K)$.  

Figure~\ref{fig:connectedsum} illustrates for an arbitrary knot $K$, the connected sum $P_n(K) \cs P_n^*(K) = P_n(K) \cs P_n(K)$ in the case of $n=5$.   Performing $n-1$ band moves in the evident way yields the $(0,n)$--cable of $K\cs K$.  Thus, if $K\cs K =0 \in \calc$, then the $n$ components of this link can be capped off with parallel copies of the slice disk for $K \cs K$, implying that  $ P_n(K) \cs P_n(K) =   0 \in \calc$.  In particular, $2 K_n = 0 \in \calc$ and $2K_n^r = 0 \in \calc$.

%%%%%%%Section%%%%%%%%%%%%%%

\section{Proof $L_n \in \ker \varphi$}\label{sec:inthekernel}

We prove a stronger statement: {\it   For all odd $n$, and for all positive integers $q$, $M_q(L_n)$ is a rational homology sphere that represents $0 \in \Theta_\Q^3$. }

The $q$--fold cyclic cover of $S^3$ branched over $K_n \cs -K_n^r$ is the same space as the $q$--fold cyclic cover of   $S^3$ branched over $K_n \cs -K_n$.  A slice disk for  $K_n \cs -K_n$ is built from $(S^3 \times I, K_n \times I)$ by removing a copy of $B^3 \times I$.  Taking the $q$--fold branched cover shows that the $q$--fold cyclic cover of $B^4$ branched over that slice disk is diffeomorphic to $M_q(K_n)^* \times I$, where $M_q(K_n)^*$ denotes a punctured copy of $M_q(K_n)$.  It remains to show that   $M_q(K_n)$ is a rational homology 3--sphere.

A formula of Fox~\cite{MR95876}  and Goeritz~\cite{MR1507011} states that the order of the first homology of $M_q(K_n)$ is given by the product of values $\Delta_{K_n}(\omega_q^i)$, where $\Delta_{K_n}(t)$ denotes the Alexander polynomial, $\omega_q$ is a primitive $q$--root of unity, and $i$ runs from 1 to $q-1$.

A result of Seifert~\cite{MR0035436}  shows that $\Delta_{K_n}(t) = \Delta_{8_{17}}(t^n) \Delta_{P_n(U)}$, where $U$ is the unknot.  We have that $P_n(U) = U$.  
The Alexander polynomial for $8_{17}$ is
\[ \Delta_{8_{17}}(t) = 1-4t+8t^2-11t^3+8t^4-4t^5+t^6.
\]
A numeric computation confirms that this polynomial does not have roots on the unit complex circle, and hence   $\Delta_{8_{17}}(t^n) $ has no roots on the unit complex circle.  From this is follows that      $\Delta_{K_n}(\omega_q^i) \ne 0$ for all $i$; thus the order of the homology of $M_q(K_n)$ is finite.   

%%%%%%%Section%%%%%%%%%%%%%%

\section{Review of twisted polynomials and $8_{17}$}\label{sec:twisted-polys}
In this section we  review twisted Alexander polynomials and their application in~\cite{MR1670420, MR1670424}  showing that $8_{17} \cs - 8_{17}^r \ne 0 \in \calc$.

Let $(X,B) \to (S^3, K)$ be the $q$--fold cyclic branched cover of a knot $K$  with $q$ a prime power.  In particular,  $X$ is  rational homology sphere.  There is a canonical surjection $\epsilon \co H_1(X - B) \to \Z$.  Suppose that $\rho \co H_1(X) \to \Z_p$ is homomorphism for some prime $p$.  Then there is an associated twisted polynomial $\Delta_{K, \epsilon, \rho}(t) \in \Q(\omega_p)[t]$.  It is well-defined, up to factors of the form $at^k$, where $a \ne 0 \in \Q(\omega_p)$.  These polynomials are  discriminants of Casson-Gordon invariants, first defined in~\cite{MR900252}. 

In the case of $K = 8_{17}$ and $q = 3$, we have $H_1(X) \cong \Z_{13} \oplus \Z_{13}$, and as a $\Z_{13}$--vector space this splits as a direct sum of a   3--eigenspace and a 9--eigenspace under the order three action of the deck transformation.   Both eigenspaces are 1--dimensional.  We denote this splitting as $E_3 \oplus E_9$. There are corresponding characters $\rho_3$ and $\rho_9$ of $H_1(X)$ onto $\Z_{13}$; these are defined as the quotient  maps onto $H_1(X) / E_3$ and onto $H_1(X)/E_9$.  We let $\rho_0$ denote the trivial $\Z_{13}$--valued character.

The values of $\Delta_{8_{17}, \epsilon, \rho_i}(t)$ are given in~\cite{ MR1670424}, duplicated here in Appendix~\ref{app:polys}. For $i = 0$ it is polynomial in $\Q[t]$.  For $i=3$ and $i=9$ it is in $\Q(\omega_{13})[t]$ and is not in  $\Q[t]$.  An essential observation is that for $8_{17}^r$, the roles of $\rho_3$ and $\rho_9$ are reversed.   All three of the polynomials  are irreducible in their respective polynomial rings, once any factors of $(1-t)$  and $t$ are removed.

In~\cite{MR1670424} the proof that $8_{17} \cs - 8_{17}^r$ is not slice comes down to the  observation that no product of the form
\[ \sigma_\delta(\Delta_{8_{17}, \epsilon, \rho_3}(t))  \sigma_\gamma(\Delta_{8_{17}, \epsilon, \rho_i}(t))
\hskip.3in \text{or} \hskip.3in 
\sigma_\delta(\Delta_{8_{17}, \epsilon, \rho_9}(t))  \sigma_\gamma(\Delta_{8_{17}, \epsilon, \rho_j}(t))
\]
is of the form  $a f(t) \overline{f(t^{-1})}(1-t)^j $ for some $f(t) \in \Q(\omega_{13})[t]$.  (That is, these products are not {\it norms} in the polynomial ring $\Q(\omega_{13})[t, t^{-1}]$, modulo powers of $(1-t)$ and $t$.)  Here $i = 0$ or $i=9$ and $j=0$ or $j=3$.  The number $a $ is in $\Q(\omega)$  and  the $\sigma_\nu$ are Galois automorphisms of $\Q(\omega_p)$ (which acts by sending $\omega_p $ to $\omega_p^\nu$).

Showing that the product of the polynomials does not factor in this way  is elementary once it is established that $\Delta_{8_{17}, \epsilon, \rho_3}(t) $ and $ \Delta_{8_{17}, \epsilon, \rho_9}(t) $ are irreducible and not Galois conjugate.

%%%%%%%Section%%%%%%%%%%%%%%

\section{Main Proof} \label{sec:mainproof}

Using the fact that $-P_n(8_{17})^r = P_n(8_{17})^r$,  the   knot $L_\alpha \cs L_\beta$  can be expanded as 
\[ P_\alpha(8_{17}) \cs   P_\alpha(8_{17} )^r \cs   P_\beta(8_{17}) \cs   P_\beta(8_{17})^r .\]
We begin by analyzing the 3--fold cover of $S^3$ branched over $P_n(8_{17})$ and assume that $3$ does not divide $n$.  This cover is $M_3(P_n(8_{17}))$ and we denote the branch set in the cover by $\widetilde{B}$.

There is the obvious separating torus $T$ in $S^3 \setminus P_n(8_{17}) $.  Since  3 does not divide $n$,  $T$ has a connected separating lift $\widetilde{T} \subset M_3(  P_n(8_{17}))$.  One sees that  $\widetilde{T}$ splits  $M_3(  P_n(8_{17}))$ into two components:   $X$, the 3--fold cyclic cover of $S^3 \setminus 8_{17}$ and $Y$,  the 3--fold cyclic branched cover of $S^1 \times B^2$, branched over $P_n$.  A simple exercise shows that since $P_n(U)$ is unknotted, $Y$ is the complement of some knot $\widetilde{J_n} \subset S^3$.

A Mayer-Vietoris argument shows that $H_1(M_3(P_n(8_{17}))) \cong \Z_{13} \oplus \Z_{13}$ and the two canonical representations $\rho_3$ and $\rho_9$ that are defined on $X$ extend trivially on $Y$ and so to $M_3(P_n(8_{17}))$.  We denote these extension $\rho_3'$ and $\rho_9'$.   Let $\epsilon'$ be the canonical surjective  homomorphism $\epsilon' \co H_1(M_3(P_n(8_{17}) ))\setminus \widetilde{B}) \to \Z$.  Restricted to $X$ we have $\epsilon'(x) = \epsilon(nx)$, where $\epsilon$ was the canonical representation to $\Z$ defined for the cover of $S^3 \setminus 8_{17}$.

In~\cite[Theorem 3.7]{MR1670420} there is a discussion of   twisted Alexander polynomials of satellite knots in $S^3$, working in the greater generality of homomorphisms to the unitary group $U(m)$.  (A map to $\Z_p$ can be viewed as a representation to $U(1)$).  The proof of that theorem, which relies on the multiplicativity of Reidemeister torsion, applies in the current setting, yielding the following lemma.

\begin{lemma}
\[
\Delta_{P_n(8_{17}), \epsilon', \rho_3'}(t) = \Delta_{ 8_{17}, \epsilon, \rho_3}(t^n)  \Delta_{\widetilde{J_n}}(t).
\]
\end{lemma}
Similar results hold for the knot  $P_n(8_{17})^r$ and for the character $\rho_9$.

As described in~\cite{MR1670424,MR1670420} Casson-Gordon theory implies that if $L_\alpha\cs L_\beta$ is slice, 
then for some 3--eigenvector or for  some 9--eigenvector the corresponding twisted Alexander polynomial is a norm; that is,  it factors as   $a t^k f(t) \overline{f(t^{-1})}$,  modulo multiples of $(1-t)$.  If it is a 3--eigenvector, the relevant polynomial is of the form 
\begin{equation}\label{eqn:product}
\Delta(t) =
\sigma_a(\Delta_{8_{17}, \epsilon, \rho_3}(t^\alpha))^{x}  \sigma_b(\Delta_{8_{17}, \epsilon, \rho_9}(t^\alpha))^{y} 
\sigma_c(\Delta_{8_{17}, \epsilon, \rho_3}(t^\beta))^{z} \sigma_d(\Delta_{8_{17}, \epsilon, \rho_9}(t^\beta))^{w} \big( \Delta_{\widetilde{J_\alpha}}(t)    \Delta_{\widetilde{J_\beta}}(t)\big)^2,
\end{equation}
where one of $x, y, z,$ or $w$ is equal to 1, and each of  the others are either 1 or  0.  

The four $\Q(\omega_{13})[t]$--polynomials  that appear here, 
\[ \Delta_{8_{17}, \epsilon, \rho_3}(t^\alpha)  , \hskip.1in  \Delta_{8_{17}, \epsilon, \rho_9}(t^\alpha)   , \hskip.1in \Delta_{8_{17}, \epsilon, \rho_3}(t^\beta)   , \hskip.1in \text{and} \hskip.1in \Delta_{8_{17}, \epsilon, \rho_9}(t^\beta),   
\]
and all their Galois conjugates are easily seen to be distinct for any pair $\alpha \neq \beta$.  The following number theoretic result implies that  there is an infinite set of primes  such that if   $\alpha   \in \calp$ and $\beta  \in \calp$,  then  no   product as given in Equation~\ref{eqn:product} can be a norm in $\Q(\omega_{13})[ t]$, proving that the connected sum $L_\alpha \cs L_\beta$ is not slice.  We will present a proof  in Appendix~\ref{app:numtheory}, 

\begin{lemma}\label{lem:factoring} {\it  Let  $f(t) \in \Z(\omega_p)[t]$ be an irreducible  monic polynomial.  If there exists     $\zeta \in \C$ such that $f(\zeta) = 0$ and $\zeta^n \ne 1$ for all $n>0$, then   the set of primes $p$ for which $f(t^p)$ is reducible is finite.}\end{lemma}

\noindent {\bf Proof of Theorem 1.} The last factor in Equation~\ref{eqn:product} involving the $\widetilde{J_n}$ is a  norm, so it can be ignored in determining if the product is a norm.

A numeric computation shows that the twisted polynomials $\Delta_{8_{17}, \epsilon, \rho_i}(t)$ for $i = 3$ and $i=9$   do not have roots on the unit circle, so   Lemma~\ref{lem:factoring} can be applied  with $\F = \Q(\omega_{13})$.   Let $\calp$ be the infinite set of  primes with the property that if  $p \in \calp$ then  $\Delta_{8_{17}, \epsilon, \rho_3}(t^p)$ and $  \Delta_{8_{17}, \epsilon, \rho_9}(t^p)$ are irreducible.    Consider the case of $x=1$ in Equation~\ref{eqn:product}.  Then, assuming that $\alpha \in \calp$ and $\beta\in \calp$,  the term $\sigma_a(\Delta_{8_{17}, \epsilon, \rho_3}) (t^\alpha)$ that appears in Equation~\ref{eqn:product} is relatively prime to the remaining factors and all the factors are irreducible, modulo powers of $t$ and $(1-t)$. Hence, the product cannot be of the form $t^k  (1-t)^j f(t)f(t^{-1})$ for any $f(t) \in \Q(\omega_{13})[t]$.  The cases of $y, z$, or $w  = 1$ are the same.

%%%%%%%Section%%%%%%%%%%%%%%

\appendix

%\vfill \eject 

\section{Factoring $ f(t^p)$}\label{app:numtheory}

In this appendix we prove  Lemma~\ref{lem:factoring}, stated in somewhat more generality as Lemma~\ref{lem:factoringmonic} below.  We first summarize some background material.  Further details can be found in any graduate textbook on algebraic number theory.

\begin{itemize}[leftmargin = 1.5em]
\item   $\A \subset \C$ denotes the ring of algebraic integers.  This is the ring consisting of all roots of monic polynomials in $\Z[t]$.

\item  For an extension field $\F / \Q$, the ring of algebraic integers in $\F$ is defined by  $\calo_\F = \F \cap \A$.

\item The property of {\it transitivity} states that if $f(t) \in \calo_{\F}[t]$ is monic and $f(\zeta) = 0$, then $\zeta \in \A$.

\item $\calo_\F^\times $ is defined to be the set of units  in $\calo_\F$.

\item The {\it norm} of an element $x \in \calo_\F$  is defined as $N(x) = \prod x_i \in \Z$ where the $x_i$ are the complex Galois conjugates of $x$.  This map  satisfies $N(xy) = N(x)N(y)$ for all $x, y \in \calo_\F$.  An element $x \in \calo_{\F}$  is in  $\calo_\F^\times $ if and only if $N(x) = \pm 1$.

\item The {\it Dirichlet Unit Theorem} states that for a finite extension $\F / \Q$, the abelian group  $\calo_\F^\times $ is finitely generated, isomorphic to $G  \oplus \Z^{r+s -1}$, where $G$ is finite cyclic, $r$ is the number of embeddings of $\F$ in $\R$, and $2s$ is the number of  non-real embeddings  of $\F$ in $\C$.

\end{itemize}

%We can now state the needed result.

\begin{lemma}\label{lem:factoringmonic} {\it  Let $\F$ be a finite extension of $\Q$ and   let $f(t) \in \calo_{\F}[t]$ be an irreducible   monic polynomial.  If there exists     $\zeta \in \C$ such that $f(\zeta) = 0$ and $\zeta^n \ne 1$ for all $n>0$, then   the set of primes $p$ for which $f(t^p)$ is reducible is finite.}\end{lemma}
 
\begin{proof}    \smallskip

\noindent{\bf Step 1.} {\it  If $f(\zeta) = 0$, then $\zeta \in   \calo_{F(\zeta)}$.}

This follows immediately from the assumption that $f(t)$ is monic.

\smallskip

\noindent{\bf Step 2.}  {\it  Suppose that $f(t) \in \F [t]$ is irreducible and $f(\zeta) = 0$.  If for some prime  $p$,  $f(t^p)$ is reducible over $\F$,  then $\zeta = \eta^p$ for some $\eta \in \calo_{\F(\zeta)}$.  
}

Let $\xi \in \C$ satisfy $\xi^p = \zeta$.  Since $f(t)$ is irreducible of degree $n$ and $f(t^p) $ is reducible, we have the degrees of extensions satisfying $[\F(\zeta) \co \F] = n   $ and $[\F(\xi) \co \F] <np   $.  It follows from the multiplicity of degrees of extensions that $[\F(\xi) \co \F(\zeta)] < p$.  

The polynomial $t^p - \zeta \in \F(\zeta)[t]$ has $\xi$ as a root.  For all $i$, $\omega_p^i \xi$ is also a root,  so $t^p - \zeta$ factors completely in $\C[t]$ as   \[t^p - \zeta = (t - \xi) (t - \omega_p \xi) \cdots (t - \omega_p^{p-1}\xi).\] 

By the degree calculation just given,   $t^p - \zeta$ has an irreducible factor $g(t)  \in \F(\zeta)[t]$ of degree $l < p$.  We can write 
$g(t) = \prod (t - \omega_p^i \xi)$ where the product is over some proper subset of $\{0, \ldots , p-1\}$.  Multiplying this out, one finds that the constant term is of the form  $\omega_p^j \xi^l \in \F(\zeta)$ for some  $j$ and $l< p$.    
Since $l$ and $p$ are relatively prime, there are integers $u$ and $v$ such that $ul + vp =1$.  Thus, $(\omega_p^j \xi^l)^u ( \xi^p)^v = \omega_p^s \xi$ for some $s$.  In particular, for some $s$, we have $\omega_p^s \xi \in \F(\zeta)$. We let $\eta = \omega_p^s \xi$ and find that $\eta^p =  (\omega_p^s)^p  \xi^p =\zeta$.  Finally, $\eta$ satisfies the monic polynomial $f(t^p)$ and thus is in $\calo_{\F(\zeta)}$.

\smallskip

\noindent{\bf Step 3.}  
{\it   The set of primes $p$ such that $\zeta = \eta^p$ for some $\eta  \in \calo_{\F(\zeta)}$ is finite.}

If  $\zeta = \eta^p$ then $N(\zeta) = N(\eta)^p$.  If $N(\zeta) \ne \pm 1$, then the set of $p$ for which $N(\zeta) = a^p$ for some integer $a$ is finite.

If $N(\zeta) = \pm 1$, then $\zeta \in \calo_{\F(\zeta)}^\times$.   Hence $\zeta$ represents a non-torsion element in a  finitely generated abelian group, and thus it has a finite number of roots.
\end{proof}
\noindent{\bf Comments.}  The argument just given is based on a summary of the proof  for the  case $\F = \Q$ presented in  MathOverflow  by Dimitrove  ~\cite{vdimitrov}.    Step (2) is a special case of the    {\it Vahlen-Capelli Theorem}, proved in the case of $\F = \Q$ by Vahlen and for fields of characteristic 0 by Capelli~\cite{capelli}. A proof for fields of finite characteristic is given in the book by R\'edei~\cite{MR0211820}.

%\vfill \eject

\section{Twisted polynomials of $8_{17}$}\label{app:polys}

Here are the three needed polynomials.  We write $\omega$ for $\omega_{13}$.\vskip.1in

\noindent $
\Delta_{8_{17},\epsilon, \rho_0}(t)  = 1 - t -34t^2 - 101t^3 -34t^4 - t^5 +t^6.
$\vskip.1in

\noindent  $ \Delta_{8_{17},\epsilon, \rho_3}(t)  / (1-t) = $\vskip.05in 

$1\   + $\vskip.05in 

$t ( 2 \omega   + 2\omega^2 + 2\omega^3 + 4\omega^4
+2 \omega^5 +2\omega^6 +\omega^7 +  \omega^8 +2 \omega^9 +4\omega^{10} +  \omega^{11} + 4\omega^{12})\ +  $\vskip.05in 

$  t^2 ( -15 \omega  -10\omega^2  -15 \omega^3 -15 \omega^4
-10 \omega^5 -10 \omega^6 -10 \omega^7 -10   \omega^8  -15 \omega^9 - 15 \omega^{10}   -10 \omega^{11} -15 \omega^{12})\ + $\vskip.05in 

$t^3( 4 \omega   +  \omega^2 + 4\omega^3 + 2\omega^4
+  \omega^5 + \omega^6 +2\omega^7 + 2 \omega^8 +4 \omega^9 +2\omega^{10} + 2 \omega^{11} +2\omega^{12})\ +  $\vskip.05in 

$t^4$
\vskip.1in 

\noindent  $ \Delta_{8_{17},\epsilon, \rho_9}(t)  / (1-t) = $\vskip.05in

$1\   + $\vskip.05in 

$t ( 6 \omega   + 5\omega^2 + 6\omega^3 + 6\omega^4
+5 \omega^5 +5\omega^6 +5\omega^7 + 5 \omega^8 +6 \omega^9 +6\omega^{10} + 5 \omega^{11} + 6\omega^{12})\ +  $\vskip.05in 

$  t^2 ( -13 \omega  -12\omega^2  -13 \omega^3 -13 \omega^4
-12 \omega^5 -12 \omega^6 -12 \omega^7 -12   \omega^8  -13 \omega^9 - 13 \omega^{10}   -12 \omega^{11} -13 \omega^{12})\ + $\vskip.05in 

$t^3( 6 \omega   + 5 \omega^2 + 6\omega^3 +6\omega^4
+  5\omega^5 +5 \omega^6 +5\omega^7 +5 \omega^8 +6 \omega^9 +6\omega^{10} +5 \omega^{11} +6\omega^{12})\ +  $\vskip.05in 

$t^4$

%%%%%%%BIBLIOGRAPHY%%%%%%%%%%%%%%
\bibliographystyle{amsplain}	
\bibliography{../../../BibTexComplete}

\end{document}